\documentclass [reqno] {amsart}

\usepackage{amsfonts}
\usepackage{amssymb}

\usepackage{color}
\usepackage{hyperref}

\newtheorem{theorem}{Theorem}[section]

\newtheorem{corollary}[theorem]{Corollary}

\theoremstyle{definition}

\newtheorem{example}[theorem]{Example}
\newtheorem{problem}[theorem]{Problem}
\theoremstyle{remark}

\newcommand{\comp}{\mathrm{comp}}
\numberwithin{equation}{section}

\newcommand{\C}{\mathbb{C}}

\begin{document}

\title[On initial boundary value problem]{On initial boundary value problem for parabolic differential operator 
with non-coercive boundary conditions}

\author[A. Polkovnikov]{Alexander Polkovnikov}

\address{Siberian Federal University,
	Institute of Mathematics and Computer Science,
	pr. Svobodnyi 79,
	660041 Krasnoyarsk,
	Russia}

\email{paskaattt@yandex.ru}

\thanks{The work was supported by the Foundation for the Advancement of Theoretical Physics and Mathematics "BASIS"}

\date{June 16, 2020}




\begin{abstract}
We consider initial boundary value problem for uniformly 2-parabolic differential operator of 
second order in cylinder domain in ${\mathbb R}^n $ with non-coercive boundary conditions. In 
this case there is a loss of smoothness of the solution in Sobolev type spaces 
compared with the coercive situation.
Using by Faedo-Galerkin method we prove that problem has unique solution in special Bochner space.
\end{abstract}

\maketitle



Initial boundary value problems for parabolic (by Petrovsky) differential operators with 
coercive boundary conditions are well studied (see, for instance, 
\cite{LadSoUr67}, \cite{LiMa72}, 
\cite{Mikh76}, \cite{Temam79}). However the problem with non-coercive 
boubdary conditions are also appeared in both theory and applications, see, for instance, 
pioneer work in this direction  \cite{ADN59}  and 
papers \cite{Ca59}, \cite{Ca60} and \cite{ShlPeich15} for such problems in the Elasticity Theory. 
Recent results in Fredholm operator equations, induced by boundary value problems for elliptic differential operators with non-coercive boundary conditions (see, for instance, \cite{PolkShla13}, \cite{PolkShla15}, \cite{PolkShla17}, \cite{ShlTark12}) allows us to apply these one for studying the parabolic problem. Consideration of such problems essentially extends variety of boundary operators, but there is a loss of regularity of the solution (see \cite{Kohn79} for elliptic case).  Namely, let $\Omega_T$ be a cylinder,
\[
\Omega_T = \Omega \times (0, T),
\]
where $\Omega$ is a bounded domain in ${\mathbb R}^n $.

Consider a second order differential operator 
$$
A  (x, t, \partial) 
= - \sum_{i, j = 1}^{n} \partial_i (a_{i,j} (x) \partial _j \cdot)
+ \sum_{j = 1}^{n} a_j (x) \partial_l
+ a_0 (x) + \frac{\partial}{\partial t}
$$
of divergence form in the domain $\Omega_T$. The coefficients $a_{i,j}$, $a_j$ are assumed to be complex-valued functions of class $L^\infty (\Omega)$.
We suppose that the matrix
$ {\mathfrak A} (x) =  \left( a_{i,j} (x) \right)_{\substack{i = 1, \ldots, n \\
		j = 1, \ldots, n}}$
is Hermitian and satisfies
\begin{equation}
\label{eq.ell.positive}
\sum_{i,j=1}^{n} a_{i,j} (x) \overline{w}_i w_j
\geq 0 \mbox{ for all } (x, w) \in \overline{\Omega} \times \C^n, 
\end{equation}
\begin{equation}
\label{eq.ell}
\sum_{i,j=1}^{n} a_{i,j} (x) \xi_i \xi_j
\geq
m\, |\xi|^2 \mbox{ for all } (x,\xi) \in \overline{\Omega} \times ({\mathbb R}^n \setminus \{ 0 \}),
\end{equation}
where $m$ is a positive constant independent of $x$ and $\xi$.
Estimate (\ref{eq.ell}) is nothing but the statement that the operator $A (x, t, \partial)$ 
is uniformly 2-parabolic.

We note that,
since the coefficients of the operator and the functions under consideration are
complex-valued,
inequalities (\ref{eq.ell.positive}) and (\ref{eq.ell}) are weaker than
\begin{equation}
\label{eq.coercive.strong}
\sum_{i,j=1}^n a_{i,j} (x)\, \overline{w}_i w_j
\geq
m\, |w|^2
\end{equation}
for all $(x,w) \in \overline{\Omega} \times ({\mathbb C}^n \setminus \{ 0 \})$.
Inequality (\ref{eq.coercive.strong}) means that correspondent Hermitian form (see form (\ref{herm.form})) is coercive.

Consider boundary operator of Robin type: 
$$
B(x,\partial) = b_1 (x) \sum_{i,j=1}^{n} a_{i,j} (x)\, \nu_i \partial_j + b_0 (x),
$$
where $b_0$, $b_1$ are bounded functions  on $\partial \Omega$ and 
$\nu (x) = (\nu_1 (x), \ldots, \nu_{n} (x))$
is the unit outward normal vector of $\partial \Omega$ at $x \in \partial \Omega$. Let $ S $ be an open connected subset of $\partial \Omega$ with piecewise smooth boundary $\partial S$.
We allow the function $b_1 (x)$ to vanish on  $S$.
In this case we assume that $b_0 (x)$ does not vanish for $x\in S$.

Consider now the following mixed initial-boundary problem in a bounded 
domain $\Omega_T$ with Lipschitz boundary $\partial \Omega_T$. 

\begin{problem} \label{pr.S-D}
	Find a distribution $u(x,t) $, satisfying the problem
	$$
	\left\{ 
	\begin{array}{ccccc}
	A  (x, t, \partial) u& = & f & in & \Omega_T ,\\
	B(x,\partial)u & = & 0 & on & \partial \Omega \times ( 0, T), \\
	u(x,0) & = & u_0 & on & \Omega. 
	\end{array}
	\right.
	$$
	with given data $f\in \Omega_T$.
\end{problem}

For solving the problem we have to define appropriate functional spaces. Denote by $C^1 (\overline \Omega,S)$ the subspace of $C^1 (\overline \Omega)$ 
consisting of those functions whose restriction to the boundary vanishes on $\overline S$. Let 
$H^1 (\Omega,S)$ be the closure of $C^1 (\overline \Omega,S)$ in $H^1 (\Omega)$.
Since on $S$ the boundary operator reduces to $B = b_0 (x)$ and $
b_0 (x) \neq 0$ for $x \in S$, then the functions $u\in H^1 (\Omega)$ satisfying $Bu = 0$ on $\partial 
\Omega$ belong to $H^1 (\Omega,S)$. 

Split now both $a_0(x)$ and $b_0(x)$ into two parts
\[
a_0 = a_{0,0} + \delta a_0,
\]
\[
b_0 = b_{0,0} + \delta b_0,
\]
where $a_{0,0}$ is a non-negative bounded function in $\Omega$ and $b_{0,0}$ is a
such function that $b_{0,0}/b_1$ is non-negative bounded function on $S$.
Then, under reasonable assumptions, the Hermitian form
\begin{equation}\label{herm.form}
(u,v)_{+}
=  \int\limits_{\Omega} \sum_{i,j=1}^n a_{i,j} \partial_j u \overline{\partial_i v}\, dx
+ (a_{0,0} u, v)_{L^2 (\Omega)} + (b_{0,0}/b_1\, u, v)_{L^2 (\partial \Omega \setminus S)}
\end{equation}
defines the scalar product on $H^{1}(\Omega,S)$. Denote by $H^{+} (\Omega)$ the completion of the space $H^1 (\Omega,S)$
with respect to the corresponding norm $\|\cdot\|_+$. From now on we assume that the space $H^+ (\Omega)$ is continuously embedded into the Lebesgue
space $ L^2 (\Omega) $, i.e. there is a constant $c>0$,  independent of $u$, such that 
\[
\| u \|_{L^2 (\Omega)} \leq c\| u \|_{+}\ \mbox{for all}\ u\in H^{+} (\Omega).
\]
It is true, if there exist a positive constant $c_1$ such that 
\[
a_{0,0} \geq c_1 \mbox{ in } \Omega.
\]
Actually we can get more subtle embedding for the space $H^+ (\Omega)$.
\begin{theorem}
	\label{t.emb.half}
	Let the coefficients $a_{i,j}$ be $C^\infty$ in a neighbourhood of the closure of $\Omega$, 
	inequalities (\ref{eq.ell.positive}), (\ref{eq.ell}) hold and
	\begin{equation}
	\label{eq.b}
	\frac{b_{0,0}}{b_1} \geq c_2 \mbox{ at } \partial \Omega \setminus S,
	\end{equation}
	with some constant $c_2 > 0$.
	Then the space $H^+ (\Omega)$ is continuously embedded into $H^{1/2-\varepsilon} (\Omega)$
	for any $\varepsilon  > 0$ if there is a positive constant $c_1$, such that
	\begin{equation}
	\label{eq.a}
	a_{0,0} \geq c_1 \mbox{ in } \Omega
	\end{equation}
	or  the operator $A$ is strongly elliptic in a neighborhood $X$ of $\overline \Omega$
	and
	\begin{equation}
	\label{eq.aa}
	\int_{X} \sum_{i,j=1}^n a_{i,j} \partial_j u \overline{\partial_i u}\, dx
	\geq
	m\, \| u \|^2_{L^2 (X)}
	\end{equation}
	for all $u \in C^\infty_{\comp} (X)$, with $m > 0$ a constant independent of $u$.
\end{theorem}
\begin{proof} See \cite[Theorem 2.5]{ShlTark12}.
\end{proof}

Of course, under coercive estimate (\ref{eq.coercive.strong}), the space 
$H^+ (\Omega)$ is continuously em\-bed\-ded into $H^{1} (\Omega)$. However, in general, the 
embedding, described in Theorem \ref{t.emb.half} is rather sharp (see 
\cite[Remark  5.1]{ShlTark12}).

The absence of coerciveness does not allows to consider 
arbitrary derivatives $\partial_j u$ for an element $u \in H^+ (\Omega)$. 
To cope with this difficulty we note that 
the matrix $ {\mathfrak A} (x) =  \left( a_{i,j} (x) \right)_{\substack{i = 1, \ldots, n \\
		j = 1, \ldots, n}}$
admits a factorisation, i.e. there is an $(m \times n)\,$-matrix ${\mathfrak D} (x)  = 
\left( {\mathfrak D}_{i,j} (x) \right)_{\substack{i = 1, \ldots, m \\
		j = 1, \ldots, n}}$ of bounded
functions in $\Omega$, such that
\begin{equation}
\label{eq.factor}
({\mathfrak D} (x))^\ast {\mathfrak D} (x)  =  {\mathfrak A} (x)
\end{equation}
for almost all $x \in D$ (see, for instance, \cite{ShlTark15}). For example, one 
could take the standard  non-negative self-adjoint square root  
${\mathfrak D} (x)= \sqrt{{\mathfrak A} (x)}$ of the matrix ${\mathfrak A} (x)$. 
Then 
$$
\sum_{i,j=1}^n a_{i,j} \partial_j u \overline{\partial_i v}\ 
=  ({\mathfrak D} \nabla v ) ^* {\mathfrak D}  \nabla u  = 
\sum_{l=1}^m \overline{{\mathfrak D}_l v}\, {\mathfrak D}_l u,
$$
for all smooth functions $u$ and $v$ in $\Omega$, where
$\nabla u$ is thought of as $n\,$-column with entries
$\partial_1 u, \ldots, \partial_n u$,
and $   {\mathfrak D}_l u:= \sum_{s=1}^n {\mathfrak D}_{l,s} (x) \partial_s u$, 
$l = 1, \ldots, m$.
From now on we may confine ourselves with first order summand of the form
$$  
\sum_{l=1}^m \tilde{a}_l (x) {\mathfrak D}_l,\quad \tilde{a}_l (x)\in L^\infty (\Omega),
$$
instead of
$$   
\sum_{j=1}^n a_{j} (x) \partial_j. 
$$ 
Since the coefficients $ \delta a_0 $, $\tilde{a}_l $ belong to $ L^\infty (\Omega) $
for all $ l=0,\dots,m $, it follows from 
Cauchy inequality that
\begin{equation}\label{b_eval1}
\left| \Big(\big(\sum_{l=1}^m \tilde{a}_l (x) {\mathfrak D}_l  + \delta a_0\big)u ,v  \Big)_{L^2 (\Omega)} 	
\right| \leq c \, \|u\|_+ \, \|v\|_+.
\end{equation}

Let now $H^{-} (\Omega)$ stand for the dual space for the space $H^{+} (\Omega)$ 
with respect to the pairing $<\cdot,\cdot>$ 
induced by the scalar product $(\cdot,\cdot)_{L^2 (\Omega)}$, 
see \cite{LiMa72}, \cite{Sche60} and elsewhere. It is a Banach space 
with the norm 
$$
\| u \|_{-}
= \sup_{\substack{v \in H^{+}   (\Omega) \\ v \ne 0}}
\frac{|(v,u)_{L^2 (\Omega)}|}{\| v \|_{+}}.
$$
The space $ L^2 (\Omega) $ is continuously embedded into $ H^-(\Omega) $, if the space $ H^+(\Omega) $ is continuously embedded into $ L^2 (\Omega) $ (see \cite{PolkShla13}). We denote by $i':L^2 (\Omega) \to H^-(\Omega) $ and $i: H^+(\Omega) \to L^2 (\Omega)$ the operators of correspondent continuously embeddings.
Thus we have a triple of the functional spaces
\[
H^+(\Omega) \overset{i}{\hookrightarrow} L^2 (\Omega) \overset{i'}{\hookrightarrow} H^-(\Omega),
\]
where each embeddings is compact under the hypothesis of Theorem \ref{t.emb.half}. 

Denote by $L^2(0,T;H^{+}(\Omega))$ the Bochner space of $L^2$-functions
\[
u(t):[0,T] \to H^+(\Omega).
\]
It is a Banach space with the norm 
\[
\|u\|_{L^2(0,T;H^{+}(\Omega))}^2 = \int_0^T \|u(t)\|^2_+ dt.
\]

Then an integration by parts in $ \Omega $ leads to a weak formulation of Problem (\ref{pr.S-D}):
\begin{problem} \label{pr.weak}
	Given $f \in L^2(0,T;H^- (\Omega))$ and $u_0 \in L^2(\Omega)$, find $u \in L^2(0,T;H^+ (\Omega))$, such that
	\begin{equation}
	\label{eq.SL.w}
	\begin{split}
	(u,v)_+ + \Big(\big(\sum_{l=1}^m \tilde{a}_l (x) {\mathfrak D}_l  + \delta a_0\big)u ,v  \Big)_{L^2 (\Omega)} 	 +  
	\frac{\partial}{\partial t}\left( u,v\right)_{L^2(\Omega)} = <f,v>
	\end{split}
	\end{equation}
	for all $v \in H^+ (\Omega)$, and
	\begin{equation}
	\label{eq.SL.b}
	u(0) = u_0.
	\end{equation}
\end{problem}

In general case the condition (\ref{eq.SL.b}) have no sense for functions $ u \in L^2 (0,T; H^+ (\Omega))$. But we will see below that function $u(t) \in L^2(0,T;H^+ (\Omega))$, satisfying (\ref{eq.SL.w}), is continuous 
and (\ref{eq.SL.b}) have a sense.

We want to apply the Faedo-Galerkin method for solving the Problem \ref{pr.weak} (see, for instance, \cite{LiMa72}, \cite{Temam79}). For this purpose we need some complete system of vectors in the space $H^+ (\Omega)$. As this system we take the set of eigenvectors of an operator, induced by the weak statement of elliptic selfadjoint problem, corresponding to the parabolic Problem \ref{pr.weak}. 
Namely, for given $f \in H^- (\Omega)$, find $u \in H^+ (\Omega)$, such that
\begin{equation}
\label{eq.SL.w11}
(u,v)_+ + \Big(\big(\sum_{l=1}^m \tilde{a}_l (x) {\mathfrak D}_l  + \delta a_0\big)u ,v  \Big)_{L^2 (\Omega)} = <f,v>.
\end{equation}

Equality (\ref{eq.SL.w11}) induces a bounded linear operator $L: H^{+} (\Omega) \to H^{-} (\Omega) $,
\begin{equation}
\label{eq.SL.w13}
(u,v)_+ + \Big(\big(\sum_{l=1}^m \tilde{a}_l (x) {\mathfrak D}_l  + \delta a_0\big)u ,v  \Big)_{L^2 (\Omega)} = <L u,v>.
\end{equation}
Denote by $L_0$ the operator $L$ in the case, when $\delta a_0 = a_l =0 $ for all $ l = 1,\dots, m $,
\begin{equation}
\label{eq.SL.w12}
(u,v)_+ = <L_0 u,v>.
\end{equation}
The operator $L_0: H^{+} (\Omega) \to H^{-} (\Omega)$ is 
continuously invertible and $\|L_0\|=\|L_0^ {-1}\|=1$ (see \cite[Lemma 2.6]{ShlTark12}). According to \cite[Lemma 3.1]{ShlTark12}, there is a system $\{h_j\}$ of eigenvectors of the compact positive selfadjoint operator $L_0^{-1}i'i : H^{+} (\Omega) \to H^{+} (\Omega) $, which is an orthonormal bases in $H^{+} (\Omega)$ and an orthogonal bases in $ L^2(\Omega) $ and $H^{-} (\Omega)$.

Let now function $u \in L^2(0,T;H^+ (\Omega))$ satisfies (\ref{eq.SL.w}).
We have from (\ref{eq.SL.w13}) 
\[
\left( \frac{\partial u}{\partial t},v\right)_{L^2(\Omega)} = < \frac{\partial u}{\partial t} ,v> = <f-L u,v>.
\]
Since $ f\in  L^2(0,T;H^- (\Omega))$ and operator $L: H^{+} (\Omega) \to H^{-} (\Omega) $ is bounded, then $ \frac{\partial u}{\partial t} \in L^2(0,T;H^- (\Omega))$. It means, that 
\begin{equation}\label{eq_nepr}
u \in C(0,T;L^2 (\Omega))
\end{equation}
(see, for instance, \cite{LiMa72} or \cite{Gaevsky}).

Using by the standard Faedo-Galerkin method (see, for instance, \cite{LadSoUr67}), \cite{LiMa72}, \cite{Temam79} we get next Theorem.
\begin{theorem}\label{t.exist}
	Under the hypothesis of Theorem \ref{t.emb.half}, the Problem \ref{pr.weak} has at least one solution $u(t)$, and, moreover,
	$ u(t)\in C(0,T;L^2 (\Omega)) $.
\end{theorem}
\begin{proof} 
	For each k we are looking for approximate solution of Problem \ref{pr.weak} on the next form
	\begin{equation}
	\label{eq.sol}
	u_k(t) = \sum_{j=1}^k g_{j k}(t)h_j,
	\end{equation} 
	and function $u_k$ satisfies
	\begin{equation}
	\label{eq.system}
	(u_k, h_i)_+ + \Big(\big(\sum_{l=1}^m \tilde{a}_l (x) {\mathfrak D}_l  + \delta a_0\big) u_k , h_i \Big)_{L^2 (\Omega)}  +
	\left( \frac{\partial u_k}{\partial t}, h_i\right)_{L^2(\Omega)} = <f, h_i>,
	\end{equation} 
	\begin{equation}
	u_k(0) = \sum_{j=1}^k \frac{(u_0, h_j)_{L^2(\Omega)}}{\|h_j\|^2_{L^2(\Omega)}} h_j,
	\end{equation} 
	for each ass $j=1,\dots,k$, where $\{h_j\}$ is the orthonormal bases in $H^{+} (\Omega)$. It means that (\ref{eq.system}) takes the form
	\begin{equation}\label{eq.system.1}
	g_{i k}(t) + \sum_{j=1}^k  \Big(\big(\sum_{l=1}^m \tilde{a}_l (x) {\mathfrak D}_l  + \delta a_0\big) h_j , h_i \Big)_{L^2 (\Omega)}  
	g_{j k}(t) + g'_{i k}(t) \| h_i\|^2_{L^2(\Omega)} = <f, h_i>,
	\end{equation}
	where $ i=1,\dots,k $. 
	It is a system of linear differential equations of first order with initial conditions
	\begin{equation}
	\label{eq.system.1.bound}
	g_{i k}(0) = \frac{(u_0, h_i)_{L^2(\Omega)}}{\|h_i\|^2_{L^2(\Omega)}},\quad i=1,\dots,k.
	\end{equation} 
	Since $ <f(t), h_i> $ is measurable function for all $ i=1,\dots,k$, then there is unique function $g_{i k}(t)$ for each $i=1,\dots,k$, satisfying (\ref{eq.system.1}) and (\ref{eq.system.1.bound}) for all $t\in\left[0,T \right] $ (see, for instance, \cite{Fillipov}). Note, as the function $ u(t) $ is complex-valued, then the functions 
	$\{g_{i k}(t)\}$ may be complex-valued too and the system (\ref{eq.system.1}) consists $ 2k $ real-valued equations in general case.
	
	Now we have to get some priori estimates for function $u_k(t)$ independent of $k$.
	Multiplying the equality (\ref{eq.system}) by the $ \overline{g_{i k}(t)} $ and summing by $ i = 1,\dots, k $ we get
	\begin{equation}
	\label{eq.apri.1}
	\|u_k\|^2_+ +   
	\left(\frac{\partial u_k}{\partial t} ,u_k\right)_{L^2(\Omega)} = <f, u_k> - \Big(\big(\sum_{l=1}^m \tilde{a}_l (x) {\mathfrak D}_l  + \delta a_0\big) u_k , u_k \Big)_{L^2 (\Omega)}.
	\end{equation}
	Hence, by the Cauchy inequality,
	\begin{equation}\label{eq.apri.2}
	2\left| \|u_k\|^2_+ +  
	\left(\frac{\partial u_k}{\partial t} ,u_k\right)_{L^2(\Omega)} \right|=
	\end{equation}  
	\[
	=2\Big| <f, u_k> - \Big(\sum_{l=1}^m \tilde{a}_l (x) {\mathfrak D}_l u_k , u_k \Big)_{L^2 (\Omega)} - (\delta a_0 u_k , u_k )_{L^2 (\Omega)} \Big| \leq 
	\]
	\[
	\leq \|f\|^2_- + \|u_k\|^2_+ + 2c_1\|u_k\|_+ \|u_k\|_{L^2 (\Omega)} + 2c_2 \|u_k\|^2_{L^2 (\Omega)}\leq 
	\]
	\[
	\leq \|f\|^2_- + \frac{3}{2}\|u_k\|^2_+  + (2c_2+2c_1^2) \|u_k\|^2_{L^2 (\Omega)}
	\]
	for some positive constants $ c_1 $ and $ c_2 $.
	As the norm $ \|u_k\|_+^2 $ is a real-valued function, we have
	\begin{equation}\label{eq.apri.3}
	2\left| \|u_k\|^2_+ +  \left(\frac{\partial u_k}{\partial t} ,u_k\right)_{L^2(\Omega)}\right|  = 
	\end{equation}
	\[
	=2\left| \|u_k\|^2_+ + \mathfrak{Re}\left(  \left(\frac{\partial u_k}{\partial t} ,u_k\right)_{L^2(\Omega)}\right) + i\mathfrak{Im}\left(  \left(\frac{\partial u_k}{\partial t} ,u_k\right)_{L^2(\Omega)} \right) \right|\geq 
	\]
	\[
	\geq 2\|u_k\|^2_+ + 2\mathfrak{Re}\left(  \left(\frac{\partial u_k}{\partial t} ,u_k\right)_{L^2(\Omega)} \right),
	\]
	where $ \mathfrak{Re}(g) $ and $ \mathfrak{Im}(g) $ denote real and imaginary parts of function $ g $ respectively.
	On the other hand, 
	\begin{equation}\label{eq.apri.12}
	\frac{d}{dt} \|u_k\|^2_{L^2(\Omega)} =  \left(\frac{\partial u_k}{\partial t} ,u_k\right)_{L^2(\Omega)} +  \left(u_k, \frac{\partial u_k}{\partial t}\right)_{L^2(\Omega)} =
	\end{equation}
	\[
	= 2\mathfrak{Re}\left( \left(\frac{\partial u_k}{\partial t} ,u_k\right)_{L^2(\Omega)}\right).
	\]
	It follows from (\ref{eq.apri.2}), (\ref{eq.apri.3}) and (\ref{eq.apri.12}) that
	\begin{equation}\label{eq.apri.4}
	\frac{1}{2}\|u_k(t)\|^2_+ + \frac{d}{dt} \|u_k(t)\|^2_{L^2(\Omega)} \leq \|f(t)\|^2_- + (2c_2+2c_1^2)\|u_k\|^2_{L^2 (\Omega)}.
	\end{equation}
	
	Now, integrating (\ref{eq.apri.4}) by $ t $ from $ 0 $ till some $ s \in(0,T) $ we get
	\[
	\frac{1}{2}\int_0^s \|u_k(t)\|^2_+ dt + \|u_k(s)\|^2_{L^2(\Omega)} - \|u_k(0)\|^2_{L^2(\Omega)} \leq 
	\]
	\[
	\leq\int_0^s \|f(t)\|^2_- dt + (2c_2+2c_1^2)\int_0^s \|u_k\|^2_{L^2 (\Omega)} dt.
	\]
	Since the sequence $\{ u_k(0) \} $ seeks to $ u_0 $  with $ k\to \infty $ strongly in $ L^2(\Omega) $, it follows from Gronwall type lemma (see \cite{Groe19} or \cite{MPF91}), that
	\[
	\|u_k(s)\|^2_{L^2(\Omega)} \leq \left( \|u_0\|^2_{L^2(\Omega)} + \int_0^T \|f(t)\|^2_- dt\right) e^{(2c_2+2c_1^2)s} .
	\]
	Hence
	\begin{equation}\label{eq.apri.5}
	\sup_{s\in [ 0,T]} \|u_k(s)\|^2_{L^2(\Omega)} \leq \left( \|u_0\|^2_{L^2(\Omega)} + \int_0^T \|f(t)\|^2_- dt\right) e^{(2c_2+2c_1^2)T}.
	\end{equation}
	The right side of (\ref{eq.apri.5}) independent of $ k $, therefore the sequence $ \{u_k(t)\} $ is bounded in $ L^\infty(0,T;L^2(\Omega))$. Then there is a subsequence  $ \{u_{k'}(t)\} $ of the sequence $ \{u_k(t)\} $ and an element $ u(t)\in L^\infty(0,T;L^2(\Omega)) $ such that $ u_{k'}(t) \to u(t) $ in the weak-* topology of $ L^\infty(0,T;L^2(\Omega)) $, namely
	\begin{equation}\label{eq.apri.7}
	\lim_{k'\to \infty} \int_0^T (u_{k'}(t) - u(t), v(t) )_{L^2(\Omega)} dt = 0
	\end{equation}
	for all $ v\in L^1(0,T;L^2(\Omega)) $.
	
	Integrating again (\ref{eq.apri.4}) by $ t $ from $ 0 $ till $ T $ and applying Gronwall type lemma we have
	\begin{equation}\label{eq.apri.6}
	\frac{1}{2}\int_0^T \|u_k(t)\|^2_+ dt + \|u_k(T)\|^2_{L^2(\Omega)}\leq 
	\end{equation}
	\[
	\leq \left( \|u_0\|^2_{L^2(\Omega)} + \int_0^T \|f(t)\|^2_- dt \right) e^{(2c_2+2c_1^2)T}.
	\]
	It means that the sequence $ \{u_k(t)\} $ is bounded in $ L^2(0,T;H^+(\Omega)) $.
	In particular, the sequence $ \{u_{k'}(t)\} $ is bounded in $ L^2(0,T;H^+(\Omega)) $ too. Hence there is a subsequence  $ \{u_{k''}(t)\} $ of the sequence $ \{u_{k'}(t)\} $ and an element $ \widetilde{u}(t)\in L^2(0,T;H^+(\Omega)) $ such that $ u_{k''}(t) \to u(t) $ in the weak topology of $ L^2(0,T;H^+(\Omega)) $,
	\begin{equation}\label{eq1}
	\lim_{k''\to \infty}\int_0^T (u_{k''}(t), v)_+ \,dt = \int_0^T (u(t), v)_+ \,dt
	\end{equation}
	for all $ v\in L^2(0,T;H^+(\Omega)) $ and
	\begin{equation}\label{eq.apri.8}
	\lim_{k''\to \infty} \int_0^T <u_{k''}(t) - \widetilde{u}(t), v(t)> dt = 0
	\end{equation}
	for all $ v\in L^2(0,T;H^-(\Omega)) $. In particular
	\begin{equation}\label{eq.apri.81}
	\lim_{k''\to \infty} \int_0^T (u_{k''}(t), v(t))_{L^2(\Omega)} dt = \int_0^T (\widetilde{u}(t), v(t))_{L^2(\Omega)} dt
	\end{equation}
	for all $ v\in L^2(0,T;L^2(\Omega)) $.
	
	From (\ref{eq.apri.7}) and (\ref{eq.apri.81}) we have
	\begin{equation}\label{eq.apri.9}
	\int_0^T (u(t) - \widetilde{u}(t), v(t))_{L^2(\Omega)} dt = 0
	\end{equation}
	for all $ v\in L^2(0,T;L^2(\Omega)) $. Hence 
	\begin{equation}\label{eq.apri.10}
	u(t) = \widetilde{u}(t) \in L^\infty(0,T;L^2(\Omega)) \cap L^2(0,T;H^+(\Omega)).
	\end{equation}
	From now on we denote by $ \{u_k(t)\} $ the subsequence $ \{u_{k''}(t)\} $.
	
	Let now $ \psi (t) $ be a scalar differentiable function on $ [0,T] $ such that
	$ \psi (T) = 0 $. Multiplying (\ref{eq.system}) by $ \psi (t) $ and integrating by $t$ we get
	\begin{equation}
	\begin{split}\label{eq.system.p1}
	\int_0^T (u_k(t), h_j)_+ \psi (t)\,dt +  \int_0^T \Big(\big(\sum_{l=1}^m \tilde{a}_l (x) {\mathfrak D}_l  + \delta a_0\big) u_k(t) , h_i \Big)_{L^2 (\Omega)}\psi (t) \,dt	 +\\ + 
	\int_0^T \left( \frac{\partial u_k(t)}{\partial t}, h_i\right)_{L^2(\Omega)} \psi (t)\,dt
	= \int_0^T<f(t), h_j> \psi (t) dt.
	\end{split}
	\end{equation}
	However
	\begin{equation}
	\begin{split}\label{eq.system.p2}
	\int_0^T \left( \frac{\partial u_k(t)}{\partial t}, h_i\right)_{L^2(\Omega)} \psi (t)\,dt =  -\int_0^T \left( u_k(t), \psi' (t) h_j\right)_{L^2(\Omega)}\, dt - \\ - ( u_k(0), h_j\psi (0))_{L^2(\Omega)},
	\end{split}
	\end{equation}
	and it follows that 
	\begin{equation}\label{eq.system.p}
	\int_0^T (u_k(t), h_j\psi (t))_+ \,dt  + \int_0^T \Big(\big(\sum_{l=1}^m \tilde{a}_l (x) {\mathfrak D}_l  + \delta a_0\big) u_k(t) , h_i \Big)_{L^2 (\Omega)}\psi (t) \,dt -
	\end{equation}
	\[
	- \int_0^T \left( u_k(t), \psi' (t) h_j\right)_{L^2(\Omega)}\, dt = ( u_k(0), h_j\psi (0))_{L^2(\Omega)} + \int_0^T<f(t), h_j> \psi (t) dt.
	\]
	Now we want to go to the limit in (\ref{eq.system.p}) with $ k \to \infty$. 
	It follows from \ref{b_eval1}, that
	\[
	\int_0^T \Big(\big(\sum_{l=1}^m \tilde{a}_l (x) {\mathfrak D}_l  + \delta a_0\big) u_k(t) , h_i \Big)_{L^2 (\Omega)}\psi (t) \,dt
	\]
	is continuous linear functional on $ L^2(0,T;H^+(\Omega)) $. Since $ u_{k}(t) \to u(t) $ with $ k\to \infty $ in the weak topology of $ L^2(0,T;H^+(\Omega)) $, we have
	\[
	\lim_{k\to \infty}\int_0^T \Big(\big(\sum_{l=1}^m \tilde{a}_l (x) {\mathfrak D}_l  + \delta a_0\big) (u_k(t) - u(t)) , h_i \Big)_{L^2 (\Omega)}\psi (t) \,dt  = 0.
	\]
	From (\ref{eq.apri.81}), (\ref{eq1}), (\ref{eq.apri.10}) and the fact that $ u_k(0) \to u_0 $ strongly in $ L^2(\Omega) $ with $ k \to \infty $ we get
	\begin{equation}\label{eq2}
	\int_0^T (u(t), h_j\psi (t))_+ \,dt + \int_0^T \Big(\big(\sum_{l=1}^m \tilde{a}_l (x) {\mathfrak D}_l  + \delta a_0\big) u(t) , h_i \psi (t)\Big)_{L^2 (\Omega)} \,dt -
	\end{equation}
	\[
	- \int_0^T \left( u(t), \psi' (t) h_j\right)_{L^2(\Omega)}\, dt = ( u_0, h_j\psi (0))_{L^2(\Omega)} + \int_0^T<f(t), h_j> \psi (t) dt.
	\]
	
	As the system $\{h_j\}_{j=1,2,\dots}$ is dense in $H^{+} (\Omega)$ and $ L^2(\Omega) $, equality (\ref{eq2}) holds by linearity and continuity for all $v\in H^{+} (\Omega)$,
	\begin{equation}\label{eq3}
	\int_0^T (u(t), v)_+ \psi (t) \,dt + \int_0^T \Big(\big(\sum_{l=1}^m \tilde{a}_l (x) {\mathfrak D}_l  + \delta a_0\big) u(t) , v \Big)_{L^2 (\Omega)}\psi (t) \,dt- 
	\end{equation}
	\[
	- \int_0^T \left( u(t), v\right)_{L^2(\Omega)} \psi' (t)\, dt = ( u_0, v)_{L^2(\Omega)} \psi (0) + \int_0^T<f(t), v> \psi (t) dt.
	\]
	In particular, if we take by $ \psi (t) $ differentiable functions with compact support in $ (0,T) $, we get 
	\begin{equation}\label{eq4}
	(u(t), v)_+ + \Big(\big(\sum_{l=1}^m \tilde{a}_l (x) {\mathfrak D}_l  + \delta a_0\big) u(t) , v \Big)_{L^2 (\Omega)} + \frac{d}{d t}\left(u(t), v\right)_{L^2(\Omega)} = <f(t),v>
	\end{equation}
	in the sense of distributions. Now we have to show that $ u(0) = u_0 $.
	Indeed, multiplying (\ref{eq4}) by $ \psi (t) $ and integrating by parts we get
	\[
	\int_0^T (u(t), v)_+ \psi (t) \,dt + \int_0^T \Big(\big(\sum_{l=1}^m \tilde{a}_l (x) {\mathfrak D}_l  + \delta a_0\big) u(t) , v \Big)_{L^2 (\Omega)}\psi (t) \,dt -
	\]
	\[
	- \int_0^T \left( u(t), v\right)_{L^2(\Omega)} \psi' (t)\, dt = ( u(0), v)_{L^2(\Omega)} \psi (0) + \int_0^T<f(t), v> \psi (t) dt.
	\]
	Comparing it with (\ref{eq3}) we get
	\[
	( u(0)-u_0, v)_{L^2(\Omega)} \psi (0) = 0
	\]
	for all $ v\in H^{+} (\Omega)$. Taking $ \psi (0)\neq 0 $ we receive $ u(0) = u_0 $.
	
	The continuity follows from (\ref{eq_nepr}). 
\end{proof}

\begin{corollary}
	Under the hypothesis of Theorem \ref{t.emb.half}, the Problem \ref{pr.weak} has one and only one solution $ u(t)\in C(0,T;L^2 (\Omega)) $, if
	\begin{equation}\label{onesol}
	\mathfrak{Re}\left( \Big(\big(\sum_{l=1}^m \tilde{a}_l (x) {\mathfrak D}_l  + \delta a_0\big) v , v \Big)_{L^2 (\Omega)}\right) \geq 0
	\end{equation}
	for all $ v\in L^2(0,T;H^+ (\Omega)) $.
\end{corollary}
\begin{proof}
	The existence of the solution follows from the Theorem \ref{t.exist}.
	Let us now show, that the solution is unique, if the condition (\ref{onesol}) and the hypothesis of Theorem \ref{t.exist} are fulfilled. Indeed, let $ v\in L^2(0,T;H^+ (\Omega)) $ is another solution of Problem \ref{pr.weak}. Denote by $ w = u-v $. Then $ w $ satisfies conditions of Problem \ref{pr.weak} and
	\[
	(w,v)_+  + \Big(\big(\sum_{l=1}^m \tilde{a}_l (x) {\mathfrak D}_l  + \delta a_0\big) w , v \Big)_{L^2 (\Omega)} + \frac{d}{d t}\left( w,v\right)_{L^2(\Omega)} = 0
	\]
	for all $v \in H^+ (\Omega)$, and $w(0) = 0$. It follows from (\ref{eq.SL.w13}), that
	\[
	\frac{\partial w}{\partial t} + L w = 0.
	\]
	Multiplying scalar it by $ w $ we have
	\[
	\|w\|^2_+  + \Big(\big(\sum_{l=1}^m \tilde{a}_l (x) {\mathfrak D}_l  + \delta a_0\big) w , w \Big)_{L^2 (\Omega)} + \left(\frac{\partial w}{\partial t}, w\right)_{L^2(\Omega)} = 0.
	\]
	As the $ \|w(t)\|^2_+ $ is a real-valued function, therefore
	\[
	\|w\|^2_+ + \mathfrak{Re}\left( \left( \frac{\partial w}{\partial t}, w\right)_{L^2(\Omega)} \right) + \mathfrak{Re}\left(\Big(\big(\sum_{l=1}^m \tilde{a}_l (x) {\mathfrak D}_l  + \delta a_0\big) w , w \Big)_{L^2 (\Omega)}\right) = 0.
	\]
	On the other hand,
	\[
	\mathfrak{Re}\left( \left( \frac{\partial w}{\partial t}, w\right)_{L^2(\Omega)} \right) = \frac{1}{2}\frac{d}{d t}\|w\|^2_{L^2(\Omega)}.
	\]
	It follows from (\ref{onesol}), that 
	\[
	2\mathfrak{Re}\left( \left( \frac{\partial w}{\partial t}, w\right)_{L^2(\Omega)} \right) = \frac{d}{d t}\|w\|^2_{L^2(\Omega)}\leq 0
	\]
	and
	\[
	\|w(t)\|^2_{L^2(\Omega)}\leq \|w(0)\|^2_{L^2(\Omega)} = 0,
	\]
	hence $ w(t) = 0 $ for almost all $ t\in[0,T] $, that completes the proof.
\end{proof}

As we already mentioned, the embedding $ H^+(\Omega) $ into $H^{1/2-\varepsilon} (\Omega)$ is rather sharp. Let us show, that the space $ L^2(0,T;H^+ (\Omega)) $ can not be continuously embedded into $ L^2(0,T;H^s (\Omega)) $ for all $ s>1/2 $.

\begin{example}
	Let $\Omega$ be a unit sphere in $\mathbb{C}$, matrix $ {\mathfrak A} (x) $ has a form
	\[
	{\mathfrak A} (x)=\left(a_{ij}(x)\right)_{\substack{i=1,2\\ j=1,2}} = 
	\begin{pmatrix}
	1 & \sqrt{-1}\\
	-\sqrt{-1} & 1
	\end{pmatrix},
	\]
	$S=\emptyset$, $a_l=0$ for $ l=0,1,\dots,m $, and $b_1= b_0 = 1$.
	Then the series 
	\[
	u _\varepsilon(z,t) = \sum_{k=0}^\infty \frac{z^kt^{k/2}}{T^{(k+1)/2}(k+1)^{\varepsilon/2}}, 
	\]
	$\varepsilon >0$, converges in $ L^2(0,T;H^+ (\Omega)) $ and 
	\[
	\|u_\varepsilon\|^2 _{L^2(0,T;H^+ (\Omega))}  = 
	\|u_\varepsilon\|^2 _{L^2(0,T;L^2 ({\mathbb S}))} = 2\pi \sum_{k=0}^\infty \frac{1}{(k+1)^{1+\varepsilon}}.
	\] 
	According to \cite[Lemma 1.4]{Shl22}
	\[
	\Big\|u_\varepsilon \Big\|^2 _{L^2(0,T;H^s ({\mathbb B}))} \geqslant \pi 
	\sum_{k=0}^\infty \frac{k^{2s-1}}{(k+1)^{1+\varepsilon}}, \, 0<s\leq 1.
	\]
	It means, that for each $s  \in (1/2,1)$ there exist $\varepsilon >0$ such that $u_\varepsilon \not \in L^2(0,T;H^s ({\mathbb B}))$. Hence, the space $L^2(0,T;H^+ ({\mathbb B}))$
	can not be continuously embedded into $ L^2(0,T;H^s (\mathbb B)) $ for all $ s>1/2 $.

\end{example}

\end{document}